\documentclass[12pt]{article} 

\usepackage{fullpage}
\usepackage[intlimits]{amsmath}
\usepackage{amssymb, amsfonts, amsthm, enumerate,amsmath}
\usepackage{commath}
\usepackage{url}
\usepackage{graphicx}
\usepackage{graphics}
\usepackage{booktabs}

\usepackage{color}
\usepackage{parskip}
\usepackage{xcolor}
\usepackage{tikz}
\usepackage{undertilde}

\usepackage{color,hyperref}
\definecolor{darkblue}{rgb}{0.0,0.0,0.3}
\hypersetup{colorlinks,breaklinks,
	linkcolor=darkblue,urlcolor=darkblue,
	anchorcolor=darkblue,citecolor=darkblue}

\theoremstyle{plain}
\newtheorem{thm}{Theorem}[section]
\newtheorem{lem}[thm]{Lemma}

\newtheorem*{cor}{Corollary}

\theoremstyle{definition}
\newtheorem{defn}{Definition}[section]

\newtheorem{exmp}{Example}[section]

\title{Characterization of Extreme Copulas}

\author{Partha Pratim Ghosh,\,\,\, Subir Kumar Bhandari\\}

\date{}

\begin{document}

	\maketitle
	
	\begin{abstract}
		In this paper our aim is to characterize the set of extreme points of the set of all $n$-dimensional copulas $(n>1)$. We have shown that a copula must induce a singular measure with respect to Lebesgue measure in order  to be an extreme point in the set of  $n$-dimensional copulas. We also have discovered some sufficient conditions for a copula to be an extreme copula. We have presented a construction of a small subset of $n$-dimensional extreme copulas such that any $n$-dimensional copula is a limit point of that subset with respect to weak convergence. The applications of such a theory are widespread, finding use in many facets of current mathematical research, such as distribution theory, survival analysis, reliability theory and optimization purposes. To illustrate the point further, examples of how such extremal representations can help in optimization have also been included.   
	\end{abstract}
	
	\section{Introduction}
	
	\hspace*{3mm} Copula models are popular in high dimensional statistical applications due to their ability to describe the dependence structure among random variables; see, e.g. \hyperlink{ref9}{[9]}, \hyperlink{ref10}{[10]}. In a situation where we need to study the influence of dependence structure on a statistical problem with given marginals of some random vector, it would be helpful if we do not have to study dependence structure over the class of all copulas- only a small class of copulas would do the job. In  \hyperlink{ref6}{[6]}, \hyperlink{ref7}{[7]}
	a special case has been considered. For a bi-variate random vector $(X,Y)$ with continuous and strictly increasing marginal distribution functions they have found a copula under which the probability of the event $\left\{X=Y\right\}$ is maximal. However, our work is applicable in more general scenarios.  Using Krein-Milman theorem (see e.g. \hyperlink{ref2}{[2]}) we can see that study of dependence structure over only the convex hull of the extreme copulas is enough. This motivates us to study and characterize the set of extreme copulas. It is also interesting from mathematical point of view. In this paper we aim to provide some properties such that any copula satisfying these properties will be an extreme copula. We also have shown that the probability induced by any extreme copula has to be singular with respect to Lebesgue measure. This shows how small the class of extreme copulas is. Finally, we have shown a particularly strong result that says we do not need to consider even the convex hull of the extreme points, but only a small subset of these extreme points in order to study the influence of dependence structure.

	\section{Preliminaries}
	
	\hspace*{3mm} First, we recall the definition  of copula. see, e.g. \hyperlink{ref5}{[5]}.\\
	\begin{defn}
		An ${n}${-dimensional copula} is an $n$-dimensional distribution function concentrated on $\left[0,1\right]^n$ whose univariate marginals are uniformly distributed on $\left[0,1\right]$.
	\end{defn}
	
	\hspace*{3mm} Clearly the set of all copulas is a convex set and it is closed (with respect to $ d_{\infty}$ metric), uniformly bounded by 1 and every copula is 1-Lipschitz continuous. Hence by Arzel\`{a}-Ascoli theorem (see, e.g. \hyperlink{ref3}{[3]}) the set of all copulas is compact. Again, as it is convex and compact by Krein-Milman theorem (see, e.g. \hyperlink{ref4}{[4]})  the set of all copulas is the closure (with respect to $d_{\infty}$ metric) of the convex hull of its extreme points. So we will study the set of extreme points.
	
	\section{Notations}
	\hspace*{3mm} For  a Borel-measurable function $f : \left[0,1\right]\mapsto\left[0,1\right]^{n-1}$ denote the graph of $f$ as
	\[
	\mathcal{G}_f=\left\{ \left(x,f(x)\right) \mid x\in \left[0,1\right] \right\}
	\]
	Write $f$ as $(g_1,g_2,\ldots,g_{n-1})$. We denote \[\mathcal{G}_f^{(i)}=\left\{(g_1(x),\ldots,g_{i-1}(x),x,g_i(x),\ldots,g_{n-1}(x)    ) \mid x\in \left[0,1\right] \right\}\]
	We denote the projection on $i$th co-ordinate  by $\pi_i$ i.e. $\pi_i:\left[0,1\right]^n\mapsto\left[0,1\right]$ such that $\pi_i\left((x_1,x_2,\ldots,x_n)\right)=x_i$. Now for any probability $P$ on $\left[0,1\right]^n$ we denote $P\circ\pi_i^{-1}$, the $i$th marginal of $P$ by $P_i$. Also, we denote  the Lebesgue measure by $\lambda$.\\
	
	Define the $n$-dimensional square $S_{\utilde{a},\varepsilon}$ as $\prod_{i=1}^{n}\left[a_i,a_i+\varepsilon\right]$.
	Now denote $R_{\utilde{a},\utilde{b}}$ as the $n$-dimensional rectangle formed by the two points $\utilde{a}$ and $\utilde{b}$ i.e. 
	${R_{\utilde{a},\utilde{b}}}=\{\utilde{x}\in \mathbb{R}^n \mid \utilde{a}\leq \utilde{x}\leq\utilde{b} \}$.\\
	
	We denote $u_j(t)=(x_1,x_2,\ldots,x_n)$ such that $x_j=t$ and $x_k=1$ for all $k\neq j$ and $v_j(t)=(y_1,y_2,\ldots,y_n)$ such that $y_j=t$ and $y_k=0$ for all $k\neq j$.

	\section{Sufficient Conditions}
	\hspace*{3mm} In this section we provide certain sufficient conditions for a copula to be an extreme copula. \\
	\begin{thm} 
		\hypertarget{theo.4.1}{}
		Let $f : \left[0,1\right]\mapsto\left[0,1\right]^{n-1}$ be a Borel measurable function and $\mu$ be a probability measure on $\left(\left[0,1\right],\mathcal{B}(\left[0,1\right])\right)$. Then there exists a unique $n$-dimensional probability supported on $\mathcal{G}_f$ whose first marginal is $\mu$.
	\end{thm}
	\begin{proof}
		At first we will show that there exists an $n$-dimensional probability supported on $\mathcal{G}_f$ whose first marginal is $\mu$. It will be enough to get such an $n$-dimensional probability $P$ on $\{S_1\times S_2 \mid S_1 \in \mathcal{B}(\left[0,1\right]), S_2 \in \mathcal{B}(\left[0,1\right]^{n-1}) \}$ as it is a semi-field for $\mathcal{B}\left(\left[0,1\right]^{n}\right)$. Define P on this semi-field as
		\[P(S_1\times S_2)=\mu(S_1\cap f^{-1}(S_2))\]
		for all $S_1 \in \mathcal{B}(\left[0,1\right])$ and $S_2 \in \mathcal{B}(\left[0,1\right]^{n-1}) $.
		Clearly it is a probability and by Caratheodory Extension Theorem  $P$ can be defined on  $\mathcal{B}\left(\left[0,1\right]^{n}\right)$. Observe that $P$ is supported on $\mathcal{G}_f$ and its first marginal is $\mu$. \\
		
		\hspace*{3mm} Now it remains to show is that $P$ is the unique probability with the required property. Let $\bar{P}$ be another  $n$-dimensional probability supported on $\mathcal{G}_f$ whose first marginal is $\mu$. Then for all $S_1 \in \mathcal{B}(\left[0,1\right])$ and $S_2 \in \mathcal{B}(\left[0,1\right]^{n-1} )  $ we have\\
		\begin{equation*}
		\begin{split}
		\bar{P}(S_1\times S_2) &= \bar{P}\left((S_1\times S_2)\cap\mathcal{G}_f \right)\\
		&= \bar{P}\left( \left\{(x,f(x)) \mid x \in S_1, f(x) \in S_2 \right\} \right)\\
		&= \bar{P}\left( \left\{(x,f(x)) \mid x \in S_1\cap f^{-1}(S_2) \right\} \right)\\
		&= \bar{P}\left( \left\{(x,y) \mid x \in S_1\cap f^{-1}(S_2), y\in \left[0,1\right]^{n-1} \right\} \right)\\
		&= \mu(S_1\cap f^{-1}(S_2))\\
		&= P(S_1\times S_2)
		\end{split}
		\end{equation*} 
		As $\{S_1\times S_2 \mid S_1 \in \mathcal{B}(\left[0,1\right]), S_2 \in \mathcal{B}(\left[0,1\right]^{n-1}) \}$ is a semi-field for $\mathcal{B}\left(\left[0,1\right]^{n}\right)$, by Caratheodory extension theorem $\bar{P}=P$. Hence $P$ is the unique probability supported on $\mathcal{G}_f$ whose first marginal is $\mu$.
	\end{proof}
	\begin{cor}
		For any Borel measurable function $f : \left[0,1\right]\mapsto\left[0,1\right]^{n-1}$ any copula supported on $\mathcal{G}_f$ is an extreme copula.
	\end{cor}
	\begin{proof}
		Let $f : \left[0,1\right]\mapsto\left[0,1\right]^{n-1}$ be a Borel measurable function. Suppose there exists a copula $C$ supported on $\mathcal{G}_f$. As the first marginal of every copula is $\mathbb{U}\left[0,1\right]$, by using  \hyperlink{theo.4.1}{Theorem 4.1} we can conclude that it is the unique copula supported on $\mathcal{G}_f$. If there exist any two copulas $C_1$ and $C_2$ such that $C$ can be written as a convex combination of them, then both $C_1$ and $C_2$ need to give zero measure on ${\mathcal{G}_f}^c$. Hence $C_1=C_2=C$, which implies $C$ is an extreme copula.
	\end{proof}
	\hspace*{3mm} For an extreme copula $C$ and for any permutation $\sigma$ on $\left\{1,2,\ldots,n\right\}$ the copula $C_\sigma$ defined as $C_\sigma(x_1,x_2,\ldots,x_n)=C(x_{\sigma(1)},x_{\sigma(2)},\ldots,x_{\sigma(n)})$ is also an extreme copula. Clearly there exists a copula supported on $\mathcal{G}_f^{(i)}$ iff  there exists a copula supported on $\mathcal{G}_f^{(j)}$
	for any $i,j \in \left\{1,2,\ldots,n\right\}$.\\
	\begin{cor}
		There are uncountably many extreme copulas on $\left[0,1\right]^n$.
	\end{cor}
	\begin{proof}
		Fix  $t\in (0,1)$. Let $L_1$ be the line joining $\utilde{0}$ and $(t,1,1,\ldots,1)$ and $L_2$ be the line joining $(t,1,1,\ldots,1)$ and $(1,0,0,\ldots,0)$. See the figure for $n=2$.\\
		\begin{center}
			\scalebox{.45}{ \includegraphics{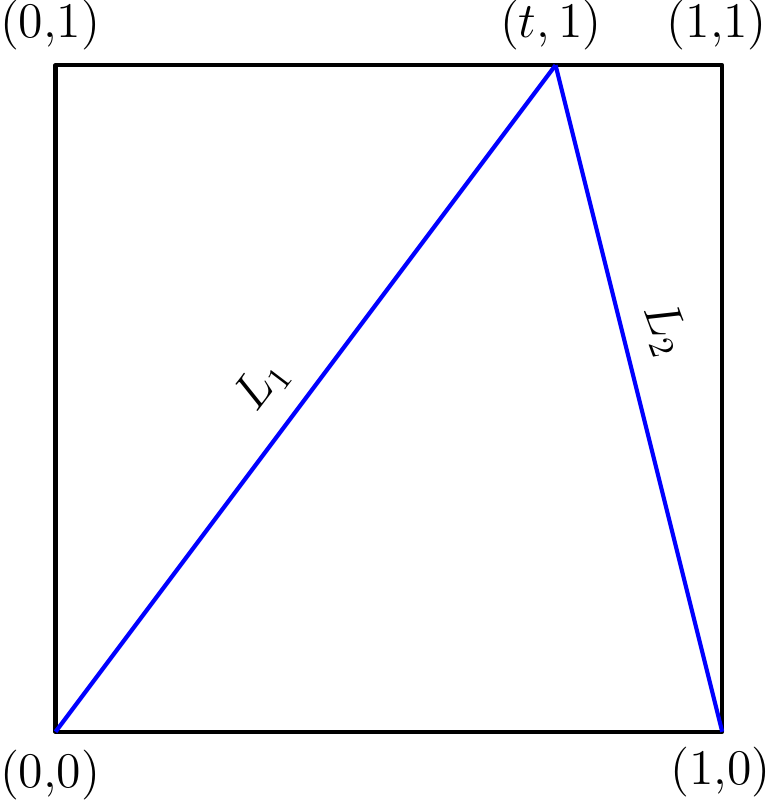}}\\
		\end{center}
		
		Let $\mathbb{U}(L_i)$ denote the uniform distribution on $L_i$. Then c.d.f. of $t\mathbb{U}(L_1)+(1-t)\mathbb{U}(L_2)$ is a copula on $L_1\cup L_2$ and hence by previous corollary it is an extreme copula. Thus for every distinct $t\in(0,1)$ we get a distinct extreme copula. And hence there are uncountably many extreme copulas. 
	\end{proof}
	
	\hspace*{3mm}We can see the above theorem talks about a small class of copulas. For example it can not  say whether the following 2-dimensional curve can be a support of an extreme copula or not.
	\begin{center}
		\scalebox{.35}{\includegraphics{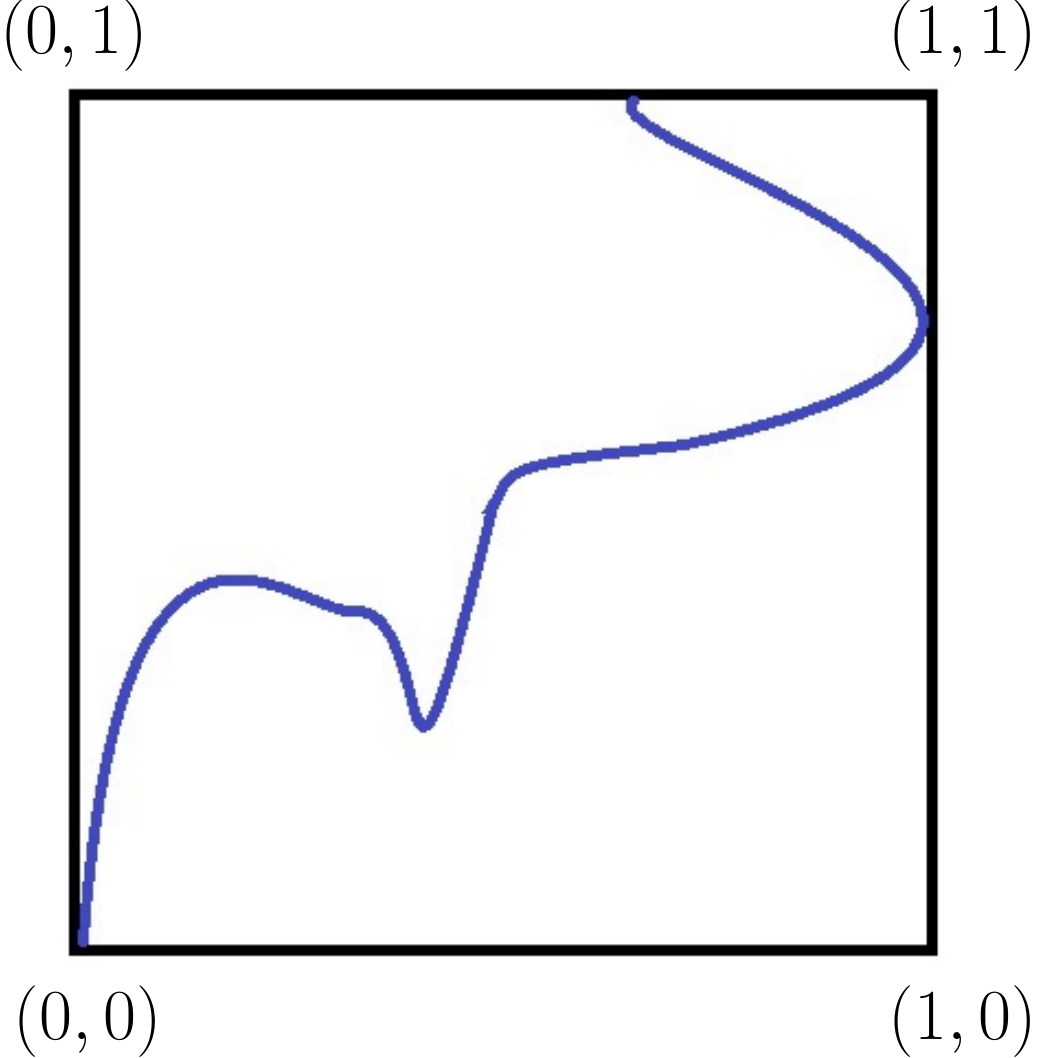}}
	\end{center}
	Therefore we need to have a more general theorem.\\
	\begin{thm}
		Let $D$ be a support of a copula such that there exist $B_1, B_2,\ldots, B_n\in \mathcal{B}(\left[0,1\right]) $ and  Borel measurable functions $f_i : \left[0,1\right]\mapsto\left[0,1\right]^{n-1}$ for $i\in \{1,2,\ldots,n\}$ such that  $\bigcup_{i=1}^n\pi_i^{-1}(B_i) \supseteq D  $ and $\pi_i^{-1}(B_i)\cap D=\pi_i^{-1}(B_i)\cap \mathcal{G}_{f_i}^{(i)}$. Then the copula will be an extreme copula. In particular it will be the unique copula supported on $D$.
	\end{thm}
	\begin{proof}
		We only need to show that there exists at most one copula supported on $D$. Let $C$, $\bar{C}$ be two copulas supported on $D$. Let $P$, $\bar{P}$ be the corresponding induced probabilities. Observe for any $S_i\in \mathcal{B} \left( \left[0,1\right]^n  \right) $, $S_i\subseteq\pi_i^{-1}(B_i)$
		\begin{equation*}
		\begin{split}
		P(S_i)= P(S_i\cap D)=P\left(S_i\cap \mathcal{G}_{f_i}^{(i)} \right)
		&= P\left(\pi_i^{-1}\left( \pi_i \left(S_i\cap \mathcal{G}_{f_i}^{(i)} \right) \right) \right)\\
		&= \lambda \left( \pi_i \left(S_i\cap \mathcal{G}_{f_i}^{(i)}\right) \right)
		\end{split}
		\end{equation*}
		Similarly
		\begin{equation*}
		\begin{split}
		\bar{P}(S_i)= \bar{P}(S_i\cap D)=\bar{P}\left(S_i\cap \mathcal{G}_{f_i}^{(i)} \right)
		&= \bar{P}\left(\pi_i^{-1}\left( \pi_i \left(S_i\cap \mathcal{G}_{f_i}^{(i)} \right) \right) \right)\\
		&= \lambda \left( \pi_i \left(S_i\cap \mathcal{G}_{f_i}^{(i)} \right) \right)
		\end{split}
		\end{equation*}
		Therefore for all $i$, ${P}(S_i)=\bar{P}(S_i)$  whenever $S_i\in \mathcal{B} \left( \left[0,1\right]^n  \right) $ and $S_i\subseteq\pi_i^{-1}(B_i)$. Now for any $A\in \mathcal{B} \left( \left[0,1\right]^n  \right) $ there exist Borel sets $A_0,A_1,\ldots,A_n$ disjoint such that  $A=\bigcup_{i=0}^n A_i$ where $A_0\subseteq D^C$ and for $i \geq 1$, $A_i\subseteq\pi_i^{-1}(B_i)$. So we have
		\[\bar{P}(A)=\sum\limits_{i=0}^n \bar{P}(A_i)=\sum\limits_{i=0}^n {P}(A_i) = P(A) \] Which implies $P=\bar{P}$ and hence $C=\bar{C}$. So, we have proved that there exists at most one copula supported on $D$ and hence it is an extreme copula.
	\end{proof}
	\hspace*{3mm} This theorem overcomes some of the drawbacks of the previous theorem. For example it can say that if there exists a copula supported on  the following 2-dimensional curve then it has to be an extreme copula.\\
	\begin{center}
		\includegraphics[scale=.5]{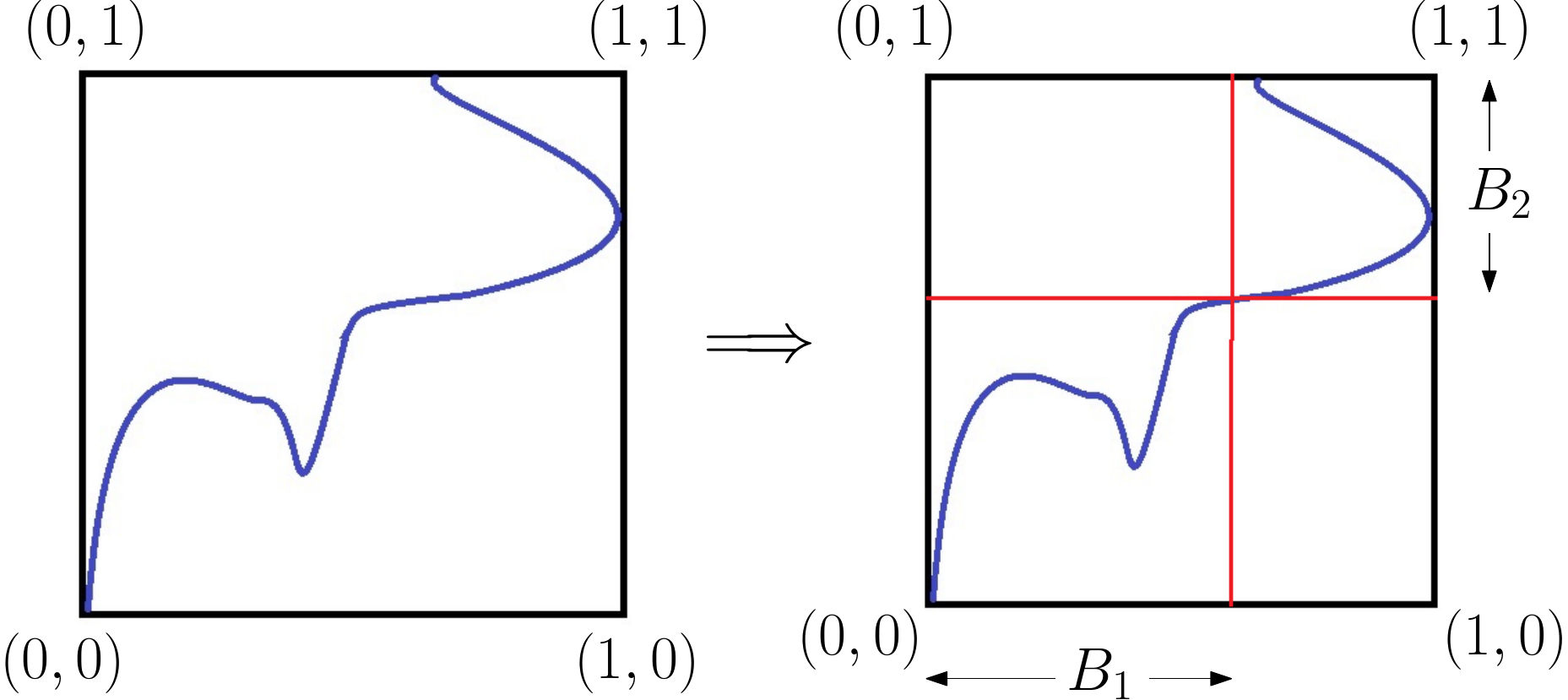}
	\end{center}  
	Now we will give some examples of extreme copulas.\\
	
	\begin{exmp}
		Let $0=t_1<t_2<\cdots<t_m<t_{m+1}=1$.
		Let $L_i$ be any interior diagonal of the $n$-dimensional rectangle $R_{v_1(t_i),u_1(t_{i+1})}$. See the figure for $n=2$.\\
		\begin{center}
			\scalebox{.5}{\includegraphics{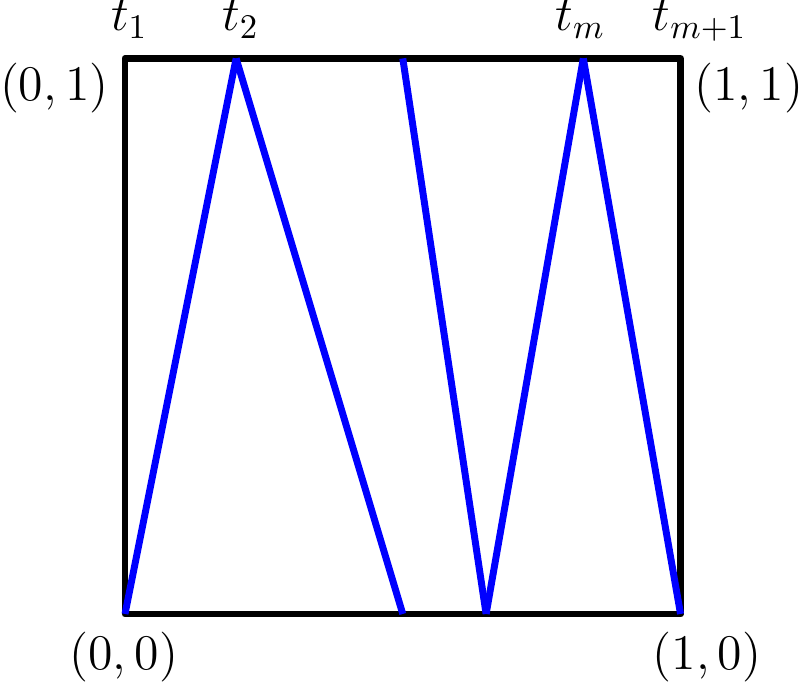}}
		\end{center}
		
		Then c.d.f. of the  probability measure $P=\sum_{i=1}^{m}(t_{i+1}-t_i)\mathbb{U}(L_i)$ is an extreme copula. To see this let
		\[
		f(x)=\sum\limits_{i=1}^m L_i(x)\cdot 1_{\left\{x\in \left[ t_i,t_{i+1} \right) \right\}}
		\]		
		Observe
		\[
		\left( \bigcup\limits_{i=1}^m L_i \right)\Delta\mathcal{G}_f \subseteq \left\{ \left( t_{i+1}, L_i\left(t_{i+1}\right) \right) \mid i \in \mathbb{N} \right\}
		\]
		But $P\left(\left\{ \left( t_{i+1}, L_i\left(t_{i+1}\right) \right) \mid i \in \mathbb{N} \right\}\right)=0$. So $P$ is supported on $\mathcal{G}_f$, $f$ measurable. 
		Observe $P_1=\sum_{i=1}^{m} (t_{i+1}-t_i)\mathbb{U}\left[t_i,t_{i+1}\right]=\mathbb{U}\left[0,1\right]$ and for $j\neq 1$ we have $P_j=\sum_{i=1}^{m} (t_{i+1}-t_i)\mathbb{U}\left[0,1\right]=\mathbb{U}\left[0,1\right]$. Therefore c.d.f. of $P$ is a copula and hence is an extreme copula.\\
	\end{exmp}
	
	\begin{exmp}
		Let $\{t_m\}$ be any non-decreasing sequence of real numbers in $\left[0,1\right]$ with $t_1=0$ and $t_0=\lim_{m\rightarrow\infty} t_m$.  Let $L_i$ be any interior diagonal of the $n$-dimensional rectangle $R_{v_1(t_i),u_1(t_{i+1})}$ and $L_0$ be any interior diagonal of the $n$-dimensional rectangle $R_{v_1(t_0),\utilde{1}}$. Then by similar approach one can get that the c.d.f. of the  probability measure $P=(1-t_0)\mathbb{U}(L_0)+\sum_{i=1}^{\infty}(t_{i+1}-t_i)\mathbb{U}(L_i)$ is an extreme copula.\\
	\end{exmp}
	
	\begin{exmp}[\textbf{Permutation copula of order} $\mathbf{m}$]
		Draw grids of size $\frac{1}{m}$ in $\left[0,1\right]^n$. Let $\sigma_k$ for $k=2,3,\ldots,n$ are permutations on $\{0,1,\ldots,m-1\}$. Define $\sigma_1$ as the identity permutation.\\

		Let $L_i$ be an interior diagonal of $S_{(\frac{\sigma_1(i)}{m},\frac{\sigma_2(i)}{m},\ldots,\frac{\sigma_n(i)}{m}),\frac{1}{m} }$ for all $i=0,1,\ldots,m-1$. See the figure for $n=2$.\\
		\begin{center}
			\scalebox{.5}{\includegraphics{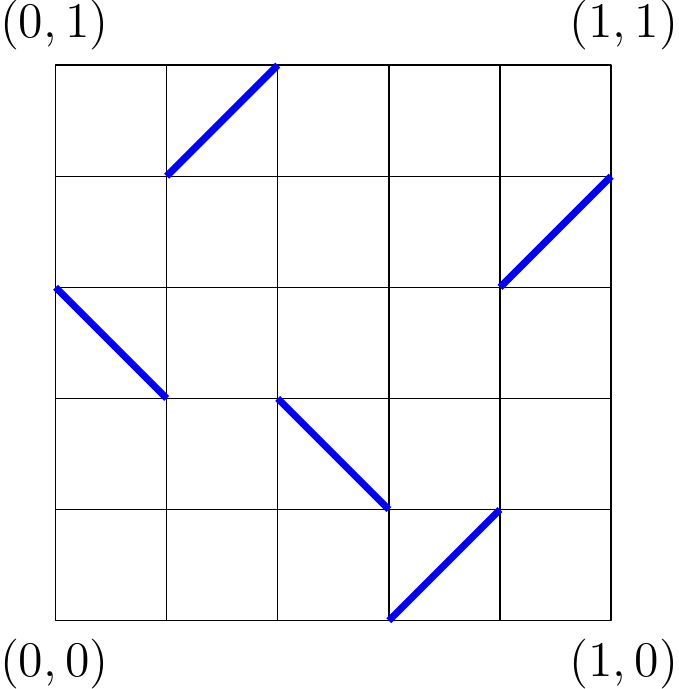}}
		\end{center}
		Let us define a probability $P$  as
		\[
		P=\sum_{i=0}^{m-1} \frac{1}{m}\cdot\mathbb{U}(L_i)
		\]
		Observe for all $k$, the $k$th marginal of $S_{(\frac{\sigma_1(i)}{m},\frac{\sigma_2(i)}{m},\ldots,\frac{\sigma_n(i)}{m}),\frac{1}{m} }$ with respect to $P$ is\linebreak $\frac{1}{m}\cdot\mathbb{U}\left(\left[\frac{\sigma_k(i)}{m},\frac{\sigma_{k}(i)+1}{m} \right]\right)$. Therefore for all $k$
		\[
		P_k= \sum_{i=0}^{m-1} \frac{1}{m}\cdot\mathbb{U}\left(\left[\frac{\sigma_k(i)}{m},\frac{\sigma_{k}(i)+1}{m} \right]\right)=\mathbb{U}\left(\left[0,1\right]\right)
		\]
		Hence c.d.f. of $P$ is a copula. Such a copula  will be called permutation copula. Let
		\[
		f(x)=\sum\limits_{i=0}^{m-1} L_i(x)\cdot 1_{\left\{x\in \left[ \frac{i}{m},\frac{i+1}{m} \right) \right\}}
		\]		
		Clearly $f$ is measurable. Observe
		\[
		P\left[ \left( \bigcup\limits_{i=0}^{m-1} L_i \right)\Delta\mathcal{G}_f  \right]\leq P\left[ \left\{ \left. \left( \frac{i+1}{m}, L_i\left(\frac{i+1}{m}\right) \right)\,\,\,\right| \,\,\, i \in \{0,1,\ldots,m-1\} \right\}  \right]=0
		\] 
		Therefore $P$ is supported on $\mathcal{G}_f$ and hence c.d.f. of $P$ is an extreme copula.\\
	\end{exmp}
	
	\begin{exmp}
		Let $C$ be a copula and $P$ be the corresponding induced probability For any $\utilde{\alpha}\in \mathbb{R}^n$ for any $A\subseteq \left[0,1\right]^n$ define
		\[
		A+\utilde{\alpha}=\left\{ \left( \{x_1+\alpha_1\},\{x_2+\alpha_2\},\ldots,\{x_n+\alpha_n\} \right) \mid \utilde{x}\in A  \right\}
		\]
		where $\{.\}$ denotes the fractional part. See the figure for $n=2$.\\
		\begin{center}
			\includegraphics[scale=1]{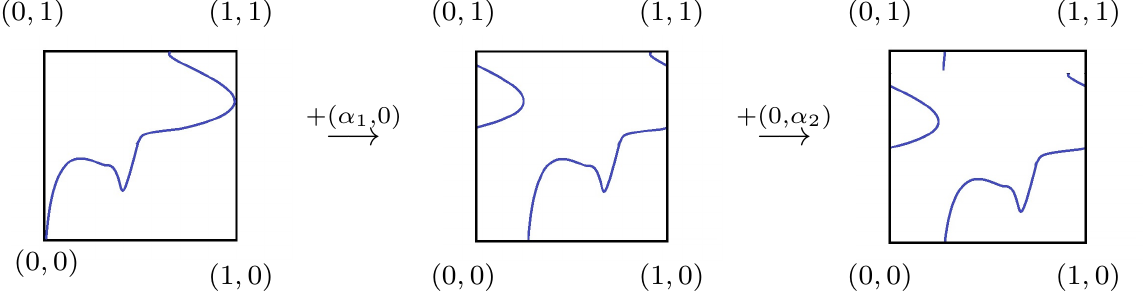}
		\end{center} \vspace*{.5cm}
		Define $P_{\utilde{\alpha}}$ as $P_{\utilde{\alpha}}(B)=P(B+\utilde{\alpha})$ for  all $B\in \mathcal{B}\left(\left[0,1\right]^n\right)$. Clearly $C_{\utilde{\alpha}}$, the c.d.f. of $P_{\utilde{\alpha}}$ is a copula. If $C$ is not an extreme copula i.e. $\exists$ copulas $C_1\neq C_2 \neq C$ such that $C=\frac{1}{2}\left(C_1+ C_2\right)$. This implies $C_{\utilde{\alpha}}=\frac{1}{2}\left((C_1)_{\utilde{\alpha}}+ (C_2)_{\utilde{\alpha}}\right)$. Hence $C_{\utilde{\alpha}}$ is not an extreme copula. Again, as $C=\left(C_{\utilde{\alpha}}\right)_{-\utilde{\alpha}}$. In a similar vein we can say $C_{\utilde{\alpha}}$ is not an extreme copula implies $C$ is not an extreme copula. Hence $C$ is  an extreme copula iff $C_{\utilde{\alpha}}$ is  an extreme copula.\\
	\end{exmp}
	\begin{exmp}
		Let $C$ be a copula and $P$ be the induced probability. Fix a co-ordinate $i$ and two disjoint intervals $\left[a,a+\delta\right],\,\left[b,b+\delta\right]\subseteq\left[0,1\right]$. Define a transformation $T$ as
		\[
		T(\utilde{x})=\left\{ 
		\begin{array}{ll}
		\left(x_1,\ldots,x_{i-1},x_i-a+b,x_{i+1},\ldots,x_n\right)& \mbox{, if } \utilde{x}\in \pi_i^{-1}\left(\left[a,a+\delta\right] \right)\\
		\left(x_1,\ldots,x_{i-1},x_i-b+a,x_{i+1},\ldots,x_n\right)& \mbox{, if } \utilde{x}\in \pi_i^{-1}\left(\left[b,b+\delta\right] \right)\\
		\utilde{x}&\mbox{, otherwise}
		\end{array}
		\right.
		\]
		\vspace*{.5cm}
		\begin{center}
			\includegraphics[scale=1]{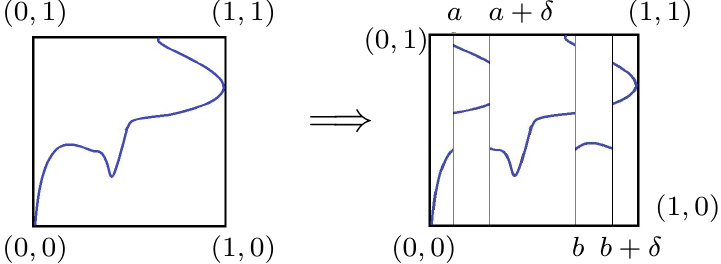}
		\end{center}
		Let $P'$ be the probability induced by $T$. As $T=T^{-1}$ we can say $P'=P\circ T$. Let the  $C'$ be c.d.f. of $P'$. Clearly $C'$ is a copula and it can be easily shown  that $C$ is an extreme copula iff $C'$ is an extreme copula.\\
	\end{exmp}
	\hspace*{3mm}Starting from the copula supported on any interior diagonal of $\left[0,1\right]^n$ if one applies the transformations mentioned in the above examples, he will always get a copula supported on a graph of a Borel measurable function. Now we are going to show that if a copula is supported on  graph of a Borel measurable function then the function has to be measure preserving in each co-ordinate with respect to Lebesgue measure.
	Conversely for any Borel measurable function  from $\left[0,1\right]$ to $\left[0,1\right]^{n-1}$, there exists a copula supported on its graph if it is measure preserving in each co-ordinate with respect to Lebesgue measure.\\
	
	\begin{thm}
		For any Borel measurable function  $f : \left[0,1\right]\mapsto\left[0,1\right]^{n-1}$, $\exists$ a copula supported on $\mathcal{G}_f$ iff $f$ is measure preserving in each co-ordinate with respect to $\lambda$. In that case the copula will be an extreme copula
	\end{thm}
	\begin{proof}
		Let $f : \left[0,1\right]\mapsto\left[0,1\right]^{n-1}$ be a Borel measurable function. Suppose there is a copula $C$  supported on $\mathcal{G}_f$. Let $P$ be the probability induced by $C$. write $f$ as $(f_2,f_3,\ldots,f_n)$. Then for all $i$, for all $B\in\mathcal{B}(\left[0,1\right])$
		\[
		\begin{split}
			\lambda\left[f_i^{-1}(B)\right] &= P\left[\pi_1^{-1}\left(f_i^{-1}(B)\right)\right]\\
			&= P\left[\pi_1^{-1}\left(f_i^{-1}(B)\right)\cap \mathcal{G}_f \right]\\
			&= P\left[ \left\{  (x,f(x)) \mid x\in f_i^{-1}(B)   \right\} \right]\\
			&= P\left[ \left\{  (x,f(x)) \mid f_i(x)\in B   \right\} \right]\\
			&= P\left[\pi_i^{-1}(B)\cap \mathcal{G}_f \right]\\
			&= P\left[ \pi_i^{-1}(B) \right]\\
			&= \lambda\left[B\right] 
		\end{split}
		\]
		Hence $f$ is measure preserving in each co-ordinate with respect to $\lambda$.\\
		\begin{center}
			\includegraphics[scale=.4]{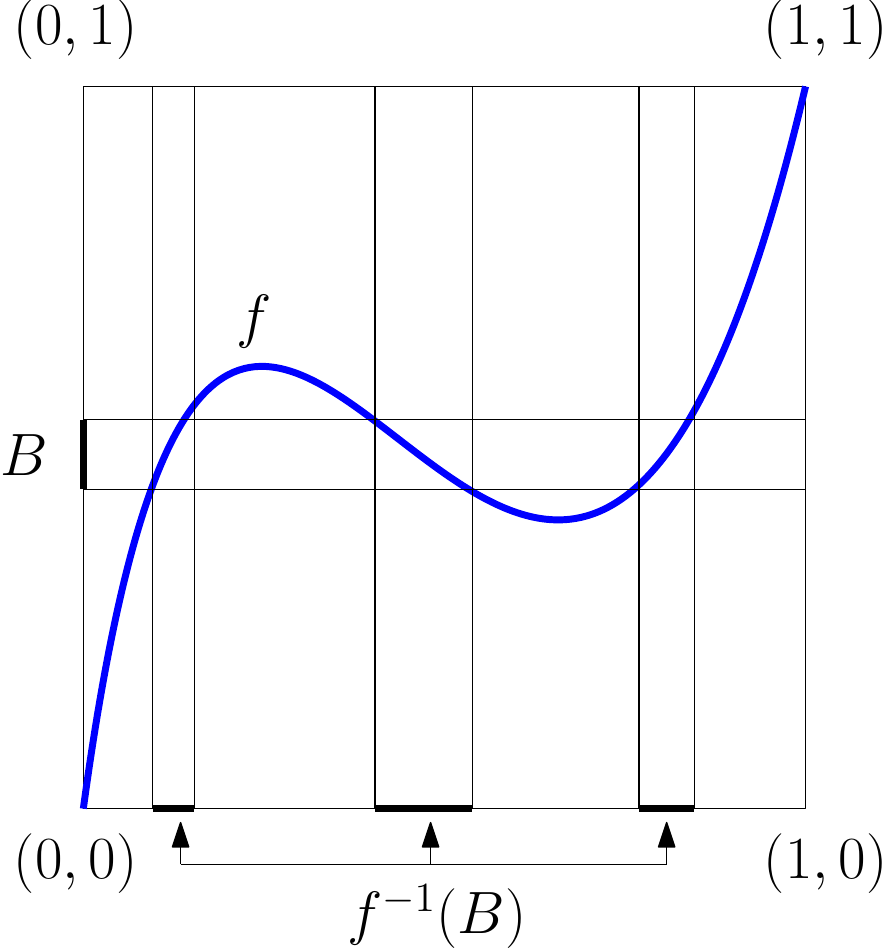}
		\end{center}
		\hspace*{3mm}Now suppose $f : \left[0,1\right]\mapsto\left[0,1\right]^{n-1}$ be a Borel measurable function which is measure preserving in each co-ordinate with respect to $\lambda$. Define a probability $P$ as \[P(B) = \lambda\left[\pi_1(B\cap\mathcal{G}_f)\right]\] for all $B\in\mathcal{B}(\left[0,1\right])$. Clearly $P$ is supported on $\mathcal{G}_f$. Let $P_i$ be its $i$th marginal. Observe
		\[
		\begin{split}
			P_1 (B)= P\left[\pi_1^{-1}(B)\right] &=P\left[\pi_1^{-1}(B)\cap \mathcal{G}_f  \right]\\
			&= \lambda \left[ \pi_1\left( \pi_1^{-1}(B)\cap \mathcal{G}_f  \right) \right]\\
			&= \lambda \left[B\right]
		\end{split}
		\]
		For all $i\neq 1$
		\[
		\begin{split}
			P_i (B)= P\left[\pi_i^{-1}(B)\right] &=P\left[\pi_i^{-1}(B)\cap \mathcal{G}_f  \right]\\
			&= P\left[\pi_1^{-1}\left(f_i^{-1}(B)\right)\cap \mathcal{G}_f \right]\\
			&= P\left[\pi_1^{-1}\left(f_i^{-1}(B)\right)\right]\\
			&= P_1\left[f_i^{-1}(B)\right]\\
			&= \lambda\left[f_i^{-1}(B)\right]\\
			&= \lambda \left[B\right]
		\end{split}
		\]
		Therefore c.d.f. of $P$ is a copula.\\
		
		\hspace*{3mm}Clearly as the copula is supported on a graph of a Borel measurable function it is an extreme copula.	
	\end{proof}
	
	\section{Necessary Conditions}
	\hspace*{3mm}A trivial necessary condition for a set $D\subseteq \left[0,1\right]^n$ to be the support of a copula is that for any $B\in \mathcal{B} \left( \left[0,1\right]^n  \right)$ satisfying $B\cap D= \pi_i^{-1}\left(\pi_i(B)\right)\cap D=\pi_j^{-1}\left(\pi_j(B)\right)\cap D $ we must have $\lambda\left(\pi_i(B)\right)=\lambda\left(\pi_j(B)\right)$.
	Here using Lebesgue decomposition theorem (see, e.g. \hyperlink{ref1}{[1]}) we are going to show that probability induced by any extreme copula has to be singular with respect to Lebesgue measure. To prove this we need to prove the following lemma.\\
	\begin{lem}
		\hypertarget{lemma.5.1}{}
		Let $F$ be a copula and $P_F$ be the induced probability.  Let the Lebesgue decomposition be \[ P_F (A)=\int_A f\,d\lambda^n\, +\mu(A)\]
		Where $f$ is non-negative, $A\in \mathcal{B}\left(\left[0,1\right]^n\right)$ and $\mu\perp\lambda^n$. Suppose there exists a closed square $S_{\utilde{a},\varepsilon}=\prod_{i=1}^{n}\left[a_i,a_i+\varepsilon\right]$ $ \subseteq \left[0,1\right]^n$  such that 
		\[\lambda^n\left[f^{-1}\left( \{0\}  \right)\cap S_{\utilde{a},\varepsilon} \right] <\frac{1}{4}\cdot\lambda^n\left[ S_{\utilde{a},\varepsilon}  \right] =\frac{\varepsilon^n}{4}  \]
		Then $F$ can not be an extreme copula.
	\end{lem} 
	\begin{proof}
		We will get a $g: \left[0,\frac{\varepsilon}{2}\right]\times\left[0,\frac{\varepsilon}{2}\right] \times \left[0,\varepsilon\right]^{n-2}\mapsto\mathbb{R}$ which is not 0 almost everywhere such that $h_1$ and $h_2$ defined below are non-negative. Define for $i=1,2$
		\[
		h_i(\utilde{x})=\left\{ \begin{array}{l}
		f(\utilde{x})  \mbox{\hspace*{1.7cm} if } \utilde{x}\in S_{\utilde{a},\varepsilon}^C\\[.35cm]
		
		f(\utilde{x}) +(-1)^ig(\utilde{x}-\utilde{a})\\
		\mbox{\hspace*{2.5cm} if }   \utilde{x}\in \left[a_1,a_1+\frac{\varepsilon}{2}\right]\times\left[a_2,a_2+\frac{\varepsilon}{2}\right]\times\prod\limits_{i=3}^n\left[a_i,a_i+\varepsilon\right]\\[.35cm]
		f(\utilde{x}) +(-1)^{i+1}g(\utilde{x}-(a_1+\frac{\varepsilon}{2},a_2,a_3,\ldots,a_n)) \\
		\mbox{\hspace*{2.5cm} if  }   \utilde{x}\in \left[a_1+\frac{\varepsilon}{2},a_1+\varepsilon\right]\times\left[a_2,a_2+\frac{\varepsilon}{2}\right]\times\prod\limits_{i=3}^n\left[a_i,a_i+\varepsilon\right]\\[.35cm]
		
		f(\utilde{x}) +(-1)^{i+1}g(\utilde{x}-(a_1,a_2+\frac{\varepsilon}{2},a_3,\ldots,a_n))\\
		\mbox{\hspace*{2.5cm} if }   \utilde{x}\in \left[a_1,a_1+\frac{\varepsilon}{2}\right]\times\left[a_2+\frac{\varepsilon}{2},a_2+\varepsilon\right]\times\prod\limits_{i=3}^n\left[a_i,a_i+\varepsilon\right]\\[.35cm]
		f(\utilde{x}) +(-1)^{i}g(\utilde{x}-(a_1+\frac{\varepsilon}{2},a_2+\frac{\varepsilon}{2},a_3,\ldots,a_n))\\
		\mbox{\hspace*{2.5cm} if }   \utilde{x}\in \left[a_1+\frac{\varepsilon}{2},a_1+\varepsilon\right]\times\left[a_2+\frac{\varepsilon}{2},a_2+\varepsilon\right]\times\prod\limits_{i=3}^n\left[a_i,a_i+\varepsilon\right]		
		\end{array}
		\right.
		\]
		
		Then we will define the copulas $H_1$ and $H_2$ such that 
		\[P_{H_1}(A)=\int_A h_1\,d\lambda^n\, +\mu(A) \mbox{\hspace{2cm}} P_{H_2}(A)=\int_A h_2\,d\lambda^n\, +\mu(A)\]
		where $P_{H_1}$ and $P_{H_2}$ are probabilities induced by $H_1$ and $H_2$ respectively. Clearly $H_1$ and $H_2$ are copulas. As $f=\frac{1}{2}(h_1+h_2)$, hence $P_F=\frac{1}{2}(P_{H_1}+P_{H_2})$.  Therefore $F=\frac{1}{2}(H_1+H_2)$.  As $g$ in not 0 almost everywhere $F\neq {H_1}\neq {H_2}$. Hence $F$ will not be
		an extreme copula. So it is enough to get $g$ which is not 0 almost everywhere such that both $h_1$ and $h_2$ are non-negative. For this we will define functions $f_1,\,f_2,\,f_3,\,f_4$  on $\left[0,\frac{\varepsilon}{2}\right]\times\left[0,\frac{\varepsilon}{2}\right] \times \left[0,\varepsilon\right]^{n-2}$ such that 
		\[
		\begin{split}
			f_1(\utilde{x})&= f(\utilde{x}+\utilde{a})\\
			f_2(\utilde{x})&= f\left(\utilde{x}+\left(a_1+\frac{\varepsilon}{2},a_2,a_3,\ldots,a_n\right)\right)\\
			f_3(\utilde{x})&= f\left(\utilde{x}+\left(a_1,a_2+\frac{\varepsilon}{2},a_3,\ldots,a_n\right)\right)\\
			f_4(\utilde{x})&= f\left(\utilde{x}+\left(a_1+\frac{\varepsilon}{2},a_2+\frac{\varepsilon}{2},a_3,\ldots,a_n\right)\right)
		\end{split}
		\]
		Take $g$ to be $\min\{f_1,\,f_2,\,f_3,\,f_4\}$. This ensures that both $h_1$ and $h_2$ are non-negative almost everywhere. Now we need to show $g$ is not 0 almost everywhere. Observe $g^{-1}\left(\{0\}\right)=\bigcup_{i=1}^4f_i^{-1}\left(\{0\}\right)$. 
		Therefore 

		\[
		\begin{split}
			\lambda^n\left(g^{-1}\left(\{0\}\right)\right)
			&= \lambda^n \left(\bigcup_{i=1}^4f_i^{-1}\left(\{0\}\right)\right)\\[.25cm]
			&\leq  \lambda^n\left(f^{-1}\left(\{0\}\right)\cap S_{\utilde{a},\varepsilon} \right)\\[.25cm]
			&< \frac{1}{4}\cdot\lambda^n\left( S_{\utilde{a},\varepsilon}  \right) \mbox{\hspace{1cm}(by assumption)}\\[.25cm]
			&= \frac{\varepsilon^n}{4}\\[.25cm]
			&= \lambda^n\left(\left[0,\frac{\varepsilon}{2}\right]\times\left[0,\frac{\varepsilon}{2}\right] \times \left[0,\varepsilon\right]^{n-2}\right)\\[.25cm]
		\end{split}
		\]
		Therefore $g$ is not $0$ almost everywhere. Hence $F$ can not be an extreme copula.
	\end{proof}
	\hspace{.25cm}
	\begin{thm}
		Probability induced by any extreme copula has to be singular with respect to Lebesgue measure.
	\end{thm}
	\begin{proof}
		Let $C$ be a copula and $P$ be the probability induced by $C$. Suppose $P$ is not singular with respect to $\lambda^n$. Then there exists non-negative Borel measurable $f$ which is not 0 almost everywhere and a measure $\mu\perp\lambda^n$ such that the Lebesgue decomposition of $P$ is given by:
		\[
		P(A)=\int_A f\,d\lambda^n\, +\mu(A)   
		\]
		So, there exists $ a,b $ with $0<a<b$ such that 
		\[
		\alpha=\lambda^n\left(f^{-1}\left([a,b]\right)\right)>0
		\]
		Fix $\varepsilon\in (0,\alpha)$. 
		
		By Lusin's theorem, there exists a compact set $K$ with $\lambda^n(K)>1-\varepsilon$ such that $f$ restricted to $K$ is continuous. Therefore the set 
		$ f^{-1}\left([a,b]\right)\cap K $ is closed. Denote it by $K_0$. So, we have $\lambda^n\left(K_0\right) \geq\alpha-\varepsilon>0$ \\
		
		Now draw grids of size $\frac{1}{m}$. Let $S_{\utilde{c_1},\frac{1}{m}},\, S_{\utilde{c_2},\frac{1}{m}},\,\ldots,\, S_{\utilde{c_r},\frac{1}{m}}$ be the minimum possible closed squares in the grid required to cover $K_0$.\\
		
		Define $A_\xi=\left\{ \utilde{x} \mid d( \utilde{x},A   )\leq \xi  \right\}$ where $d$ is the Euclidean distance. Then it can be easily shown that ${K_0}_{\frac{\sqrt{n}}{m}}\supseteq \bigcup_{p=1}^r S_{\utilde{c_p},\frac{1}{m}}$. Observe as $m\rightarrow \infty$, ${K_0}_{\frac{\sqrt{n}}{m}}\downarrow K_0$ as $K_0$ is closed. Therefore
		
		\[\lim\limits_{m\rightarrow\infty}\lambda^n\left({K_0}_{\frac{\sqrt{n}}{m}}\cap f^{-1}\left\{0\right\}\right)=\lambda^n\left(  K_0\cap f^{-1}\left\{0\right\}\right)\Rightarrow \lim\limits_{m\rightarrow\infty}\lambda^n\left(\bigcup_{p=1}^r S_{\utilde{c_p},\frac{1}{m}}\cap f^{-1}\left\{0\right\}\right)=0
		\]
		So we can choose $m$ large enough such that
		
		\[
		\begin{split}
			\lambda^n\left[f^{-1}\left( \{0\}  \right)\cap \left( \bigcup_{p=1}^r S_{\utilde{c_p},\frac{1}{m}} \right) \right]
			&< \frac{\alpha-\varepsilon}{4}\\[.25cm]
			&\leq \frac{1}{4}\cdot \lambda^n \left(K_0\right)\\[.25cm]
			&\leq \frac{1}{4}\cdot \lambda^n\left( \bigcup_{p=1}^r S_{\utilde{c_p},\frac{1}{m}}  \right)
	    \end{split}
	    \]  
		
		So, $\exists\, q\in\left\{1,2,\ldots,r\right\}$	such that
		\[
		\lambda^n\left[f^{-1}\left( \{0\}  \right)\cap S_{\utilde{c_q},\varepsilon} \right] \geq\frac{1}{4}\cdot\lambda^n\left[ S_{\utilde{c_q},\varepsilon}  \right]
		\]   
		
		And hence by \hyperlink{lemma.5.1}{Lemma 5.1}, $C$ can not be an extreme copula.
	\end{proof}	\vspace{.5cm}
	
	\hspace*{3mm} In the following 3-dimensional figure denote $L_1$ to be the line joining $ (0,0,\frac{1}{2}) $ and $ (\frac{1}{2},\frac{1}{2},1) $, $L_2$ to be the line joining $ (\frac{1}{2},\frac{1}{2},\frac{1}{2}) $ and $ (1,1,1) $, $L_3$ to be the line joining $ (0,\frac{1}{2},0) $ and $ (\frac{1}{2},1,\frac{1}{2}) $ and $L_4$ to be the line joining $ (\frac{1}{2},0,0) $ and $ (1,\frac{1}{2},\frac{1}{2}) $. Define a probability $P$ as $P=\sum_{i=1}^{4}\mathbb{U}(L_i)$. It can be easily shown that c.d.f. of $P$ is an extreme copula. Clearly it satisfies the necessary condition but there does not exist any $i\in\{1,2,3\}$, $B\in \mathcal{B}(\left[0,1\right]) $ and  Borel measurable functions $f : \left[0,1\right]\mapsto\left[0,1\right]^2$  such that $\pi_i^{-1}(B)\cap(\cup_{i=1}^4 L_i)=\pi_i^{-1}(B)\cap \mathcal{G}_f^{(i)}$. Hence it does not satisfy the sufficient condition.
	\begin{center}
		\includegraphics[scale=.8]{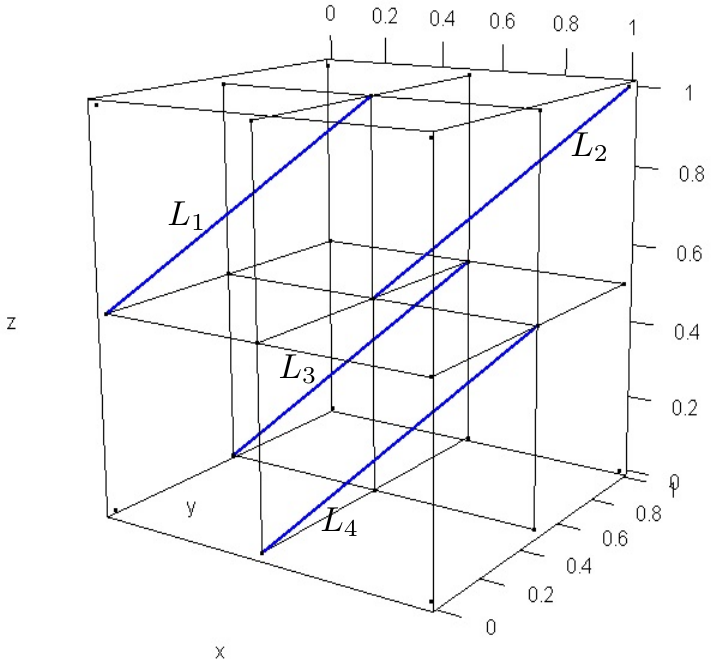}
	\end{center}
	
	\hspace*{3mm}As it has been shown above, the necessary condition and sufficient condition that we have postulated do not necessarily hold together and hence cannot be melded to give a complete characterization of extreme copulas. The initiative to find such a result that combine both might be undertaken as a future research problem. 
	
	\section{A construction of dense set using extreme copulas}
	\hspace*{3mm}We are going to show that any copula can be uniformly approximated by extreme copulas. In fact, any copula can be uniformly approximated by permutation copulas. For this we will need the following definitions.\\
	
	\begin{defn}
		A $n$-dimensional matrix of order $m_1\times m_2\times\cdots\times m_n$ is a real valued function on $\{0,1,\ldots,m_1-1\}\times\{0,1,\ldots,m_2-1\}\times\cdots\times\{0,1,\ldots,m_n-1\}$
	\end{defn}
	\vspace{.25cm}
	\begin{defn}
		A $n$-dimensional $m\times m\times\cdots\times m$ matrix $M$  is called stochastic in $k$th co-ordinate if for all $i_k\in\{0,1,\ldots,m-1\}$  \[\mathop{\sum\limits_{(r_1,r_2,\ldots,r_n)}}_{r_k=i_k}  M(r_1,r_2,\ldots,r_n) =\frac{1}{m}  \]
	\end{defn}
	\hspace*{3mm}Let us denote by $\mathcal{M}_m^n$ the set of all  $n$-dimensional $m\times m\times\cdots\times m$ matrices with non-negative entries which are stochastic in every co-ordinate.\\
	
	\begin{thm}
		Any copula can be uniformly approximated by permutation copulas.
	\end{thm}
	\begin{proof}
		Let $C$ be any copula and $P$ be the probability induced by $C$. Draw grids of size $\frac{1}{m}$ in $\left[0,1\right]^n$. Define a $n$-dimensional $m\times m\times\cdots\times m$ matrix $M$ as
		\[  M(i_1,i_2,\ldots,i_n) = P\left(S_{\left(\frac{i_1}{m},\frac{i_2}{m},\ldots,\frac{i_n}{m}\right),\frac{1}{m}}\right)  \]
		Clearly $M\in\mathcal{M}_m^n $. Define a metric $\rho$ on $\mathcal{M}_m^n$ as for all $M_1,M_2\in\mathcal{M}_m^n$
		\[\rho\left(M_1,M_2  \right)= \sup_{\left(i_1,i_2,\ldots,i_n\right)}
		\abs{ M_1\left(i_1,i_2,\ldots,i_n\right)- M_2\left(i_1,i_2,\ldots,i_n\right)} \] 
		\hspace*{3mm}Now get $N\in\mathcal{M}_m^n $ with all rational entries such that $\rho\left(N,M\right)<\varepsilon=m^{-(n+1)}$. We are going to construct a permutation copula $\bar{C}_m$ with induced probability $\bar{P}_m$ such that \\
		\[   \bar{P}_m\left(S_{\left(\frac{i_1}{m},\frac{i_2}{m},\ldots,\frac{i_n}{m}\right),\frac{1}{m}}\right) = N(i_1,i_2,\ldots,i_n)  \]\\
		If we construct it then we will have \\
		\[
		\begin{split}
			&\abs{C\left(\frac{i_1}{m},\frac{i_2}{m},\ldots,\frac{i_n}{m}\right) - \bar{C}_m\left(\frac{i_1}{m},\frac{i_2}{m},\ldots,\frac{i_n}{m}\right)    }\\[.25cm]
			&\leq \sum\limits_{(r_1,r_2,\ldots,r_n)} \abs{  P\left(S_{\left(\frac{r_1}{m},\frac{r_2}{m},\ldots,\frac{r_n}{m}\right),\frac{1}{m}}\right)  - \bar{P}_m\left(S_{\left(\frac{r_1}{m},\frac{r_2}{m},\ldots,\frac{r_n}{m}\right),\frac{1}{m}}\right) }\\[.25cm]
			&= \sum\limits_{(r_1,r_2,\ldots,r_n)} \abs{  M(r_1,r_2,\ldots,r_n) - N  (r_1,r_2,\ldots,r_n)   }\\[.25cm]
			&< m^n\varepsilon\\[.25cm]
			&= \frac{1}{m}
    	\end{split}
    	\]\\
		and therefore for all $\utilde{t}\in \left[0,1\right]^n$\\
		\[
		\begin{split}
			& \abs{C(\utilde{t})-\bar{C}_m(\utilde{t})}\\[.25cm]
			&\leq \,\abs{ C(\utilde{t}) - C\left(\frac{\lfloor mt_1 \rfloor}{m},\ldots,\frac{\lfloor mt_n \rfloor}{m}  \right)  }+
			\abs{ \bar{C}_m(\utilde{t}) - \bar{C}_m\left(\frac{\lfloor mt_1 \rfloor}{m},\ldots,\frac{\lfloor mt_n \rfloor}{m}  \right)  }\\[.25cm]
			& \mbox{\hspace{2.6cm}}+ \abs{ {C}\left(\frac{\lfloor mt_1 \rfloor}{m},\ldots,\frac{\lfloor mt_n \rfloor}{m}  \right) - \bar{C}_m\left(\frac{\lfloor mt_1 \rfloor}{m},\ldots,\frac{\lfloor mt_n \rfloor}{m}  \right)  }\\[.25cm]
			&\leq \frac{n}{m}+\frac{n}{m}+\frac{1}{m}\\[.25cm]
			&= \frac{2n+1}{m}
		\end{split}
		\]\\
		This will imply $\bar{C}_m$ converges to $C$ uniformly as $m\rightarrow\infty$. Hence $C$ can be uniformly approximated by permutation copulas. So, the only thing remains is to construct such a permutation copula $\bar{C}_m$.\\
		
		\hspace*{3mm}To do this, for all $(i_1,i_2,\ldots,i_n)$ we want to have an $n$-dimensional cube which is a subset of $S_{\left(\frac{i_1}{m},\frac{i_2}{m},\ldots,\frac{i_n}{m}\right),\frac{1}{m}}$ and length of whose edges is same as $N(i_1,i_2,\ldots,i_n)$ such that the co-ordinate projections of those cubes will be disjoint. Then  if we have a permutation copula which gives probability $N(i_1,i_2,\ldots,i_n)$ to the cube that is subset of $S_{\left(\frac{i_1}{m},\frac{i_2}{m},\ldots,\frac{i_n}{m}\right),\frac{1}{m}}$ we are done. See the figure for $n=2$\\
		
		\begin{center}
			\scalebox{.5}{\includegraphics{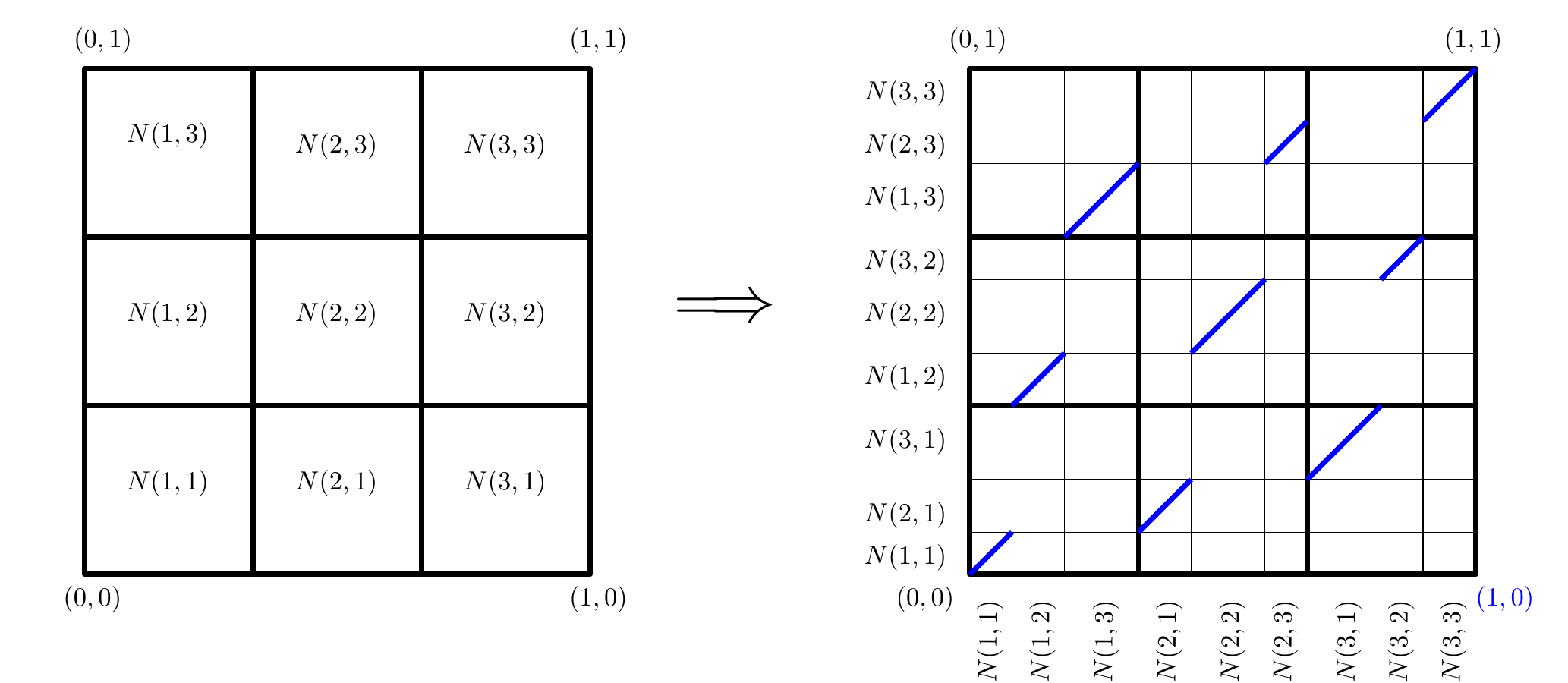}}
		\end{center}
		\hspace{.25cm}
		Observe for all $k$
		\[ \lambda\left(\left[ \frac{j}{m},\frac{j+1}{m} \right) \right) = \frac{1}{m}=  \mathop{\sum\limits_{(r_1,r_2,\ldots,r_n)}}_{r_k=j}  N(r_1,r_2,\ldots,r_n)   \]
		So we can divide $\left[ \frac{j}{m},\frac{j+1}{m} \right)$ into $m^{n-1}$ disjoint intervals as
		\[\left[ \frac{j}{m},\frac{j+1}{m} \right) = \mathop{\bigcup\limits_{(r_1,r_2,\ldots,r_n)}}_{r_k=j}  I^{(k)}_{(r_1,r_2,\ldots,r_n)} \]
		such that \[\lambda\left(I^{(k)}_{(r_1,r_2,\ldots,r_n)} \right) =N(r_1,r_2,\ldots,r_n) \] Observe \[S_{\left(\frac{i_1}{m},\frac{i_2}{m},\ldots,\frac{i_n}{m}\right),\frac{1}{m}}\supseteq \prod_{k=1}^n I^{(k)}_{(i_1,i_2,\ldots,i_n)} \]
		Let $L_{(i_1,i_2,\ldots,i_n)}$ be any one of the interior diagonals of $\overline{\left(\prod_{k=1}^{n}I^{(k)}_{(i_1,i_2,\ldots,i_n)}\right)} $. Define $\bar{P}_m$ as
		\[ \bar{P}_m=\sum_{(i_1,i_2,\ldots,i_n)}  N(i_1,i_2,\ldots,i_n)\cdot \mathbb{U}\left( L_{(i_1,i_2,\ldots,i_n)} \right) \]
		Clearly
		\[   \bar{P}_m\left(S_{\left(\frac{i_1}{m},\frac{i_2}{m},\ldots,\frac{i_n}{m}\right),\frac{1}{m}}\right) = N(i_1,i_2,\ldots,i_n)  \]
		and $\bar{C}_m$ is a copula. As $N(i_1,i_2,\ldots,i_n)$ are rationals then $\exists q\in \mathbb{N}$ such that for all $k$ for all $(i_1,i_2,\ldots,i_n) \in \{0,1,\ldots,m-1\}^n$, $\lambda\left(I^{(k)}_{(r_1,r_2,\ldots,r_n)} \right)$ is a integer multiple of $\frac{1}{q}$. Hence $\bar{C}_m$ is a permutation copula of order $q$.
	\end{proof}
	\hspace*{3mm}So we obtain that permutation copulas are dense in the space of copulas. This is an example of a convex set containing more than one point  whose extreme points are dense in that set. As a corollary we got that the set of extreme copulas are not closed under $L_{\infty}$ norm. 
	
	\section{Application}
	\hspace*{3mm} In a situation where we need to study the influence of the dependence structure on a statistical problem with given marginals of some random vector, we consider an optimization problem over the Fr\'{e}chet class $\mathcal{F}(F_1, F_2,\ldots,F_n)$ of all joint distributions with the marginals $F_1, F_2,\ldots,F_n$ (see, e.g. \hyperlink{ref8}{[8]}). For a given bounded continuous function $g:\mathbb{R}^n\mapsto\mathbb{R}$, an optimization problem over the Fr\'{e}chet class $\mathcal{F}(F_1, F_2,\ldots,F_n)$ looks like
	\[ m(g):= \sup\left\{\int g\,dF : F\in \mathcal{F}(F_1, F_2,\ldots,F_n)  \right\} \]
	Clearly we have 
	\[
	\begin{split}
	m(g)&= \sup\left\{\mathbb{E}\left( g(\utilde{X}) \right) : \utilde{X}\sim F\in \mathcal{F}(F_1, F_2,\ldots,F_n)  \right\}\\
	&= \sup\left\{\mathbb{E}\left( g\circ (F_1, F_2,\ldots,F_n)^{-1}(\utilde{U}) \right) : \utilde{U}\sim C \mbox{, where $C$ is a copula}  \right\}\\
	&= \lim\limits_{k\rightarrow\infty} \max\left\{\mathbb{E}\left( g\circ (F_1, F_2,\ldots,F_n)^{-1}(\utilde{U}) \right) : \utilde{U}\sim C_k \mbox{, a permutation copula of order $k$}  \right\}
	\end{split}
	\]
	where $(F_1, F_2,\ldots,F_n)^{-1}(\utilde{U}) = (F_1^{-1}(U_1), F_2^{-1}(U_2), \ldots,F_n^{-1}(U_n))$.\\
	
	\hspace*{3mm} In special cases in which we need to maximize probability of the event $\left\{X=Y\right\}$ (i.e., $g=1_{\left\{x=y\right\}}$) where $X,\,Y$ are random variables with distribution functions $F_X,\, F_Y$, we can construct a function $g_{\varepsilon}$ as
	\[  g_{\varepsilon}(x,y)=\left\{ \begin{array}{ll}
	1& \mbox{, if $x=y$}\\
	1-\frac{|x-y|}{\varepsilon}&\mbox{, if $|x-y|<\varepsilon$}\\
	0& \mbox{, otherwise}
	\end{array}  \right.  \]
	If $P_F$ is the probability induced by the joint distribution function $F$ we have 
	\[
	\begin{split}
	m(g)&=\sup\left\{P_F(X=Y) : F\in \mathcal{F}(F_X, F_Y)  \right\}\\
	&= \lim\limits_{\varepsilon\rightarrow 0}\lim\limits_{k\rightarrow\infty} \max\left\{\mathbb{E}\left( g_{\varepsilon}\circ (F_X, F_Y)^{-1}(\utilde{U}) \right) : \utilde{U}\sim C_k \mbox{, a  permutation copula of order $k$}  \right\}
	\end{split}
	\]
	\hspace*{3mm} We can use these results to simulate in computer the approximate value of $m(g)$ to solve the optimization problem.
	\section{References}
	
	\begin{description}
		\item[[\hypertarget{ref1}{1}]] \textsc{R. B. Ash}: Probability and Measure Theory, 2nd edn. Harcourt Academic Press, 2000.
		\item[[\hypertarget{ref2}{2}]] \textsc{Dunford, Schwartz}: Linear Operators, Part I: General Theory, Wiley, 1958.
		\item[[\hypertarget{ref3}{3}]] \textsc{George F. Simmons}: Introduction to Topology and Modern Analysis, Robert E Krieger Publishing Company, 1983.
		\item[[\hypertarget{ref4}{4}]] \textsc{Barry Simon}: Convexity: an Analytic Viewpoint (Cambridge Tracts in Mathematics 187), Cambridge University Press, 2011.
		\item[[\hypertarget{ref5}{5}]] \textsc{Fabrizio Durante, Carlo Sempi}: Principles of Copula Theory. CRC Press,  2015.
		\item[[\hypertarget{ref6}{6}]] \textsc{Fabrizio Durante, J. Fern\'{a}ndez S\'{a}nchez, W. Trutschnig}: On the Singular Components of a Copula, J. Appl. Prob. 52, 1175-1182, 2015.
		\item[[\hypertarget{ref7}{7}]] \textsc{J.-F. Mai, M. Scherer}: Simulating from the copula that generates the maximal probability for a
		joint default under given (inhomogeneous) marginals. In Topics in Statistical Simulation, eds V. B. Melas et al.,
		Springer, New York, pp. 333-341, 2014. 
		\item[[\hypertarget{ref8}{8}]] \textsc{Giovanni Puccetti, Ruodu Wang}: Extremal Dependence Concepts. Statistical Science, vol.30, No.4, 485-517, 2015.
		\item[[\hypertarget{ref9}{9}]] \textsc{P. Jaworski, F. Durante, W. K. H\"{a}rdle}: Copulae in Mathematical and Quantitative Finance (Lecture Notes in Statistics 213), Springer, 2013.
		\item[[\hypertarget{ref10}{10}]] \textsc{P. Jaworski, F. Durante, W. K. H\"{a}rdle, T. Rychlik}: Copula Theory and Its Applications (Lecture Notes in Statistics 198), Springer, 2010.
		
	\end{description}
	
\end{document}